\documentclass[11pt,reqno]{amsart}

\usepackage{graphicx,color}
\usepackage{amssymb,amscd,latexsym, mathrsfs, epsfig}
\usepackage[all]{xy}
\usepackage[latin1]{inputenc}
\usepackage{geometry}
\newtheorem{satz}{Theorem}[section]

\newtheorem{bem}[satz]{Remark}
\newtheorem{kor}[satz]{Corollary}
\newtheorem{lem}[satz]{Lemma}

\newtheorem{bei}[satz]{Example}
\newtheorem{pro}[satz]{Proposition}

\newcommand{\filt}[6]{
\[
\begin{xy}
\xymatrix@R20pt@C20pt{
&\mathbb{C}^3&\\\langle #1,#2 \rangle\ar[ru]&\langle #3,#4\rangle\ar[u]&\langle #5,#6\rangle\ar[lu]\\\langle #1\rangle\ar[u]&\langle #3\rangle\ar[u]&\langle #5\rangle\ar[u]}
\end{xy}
\]
}

\newcommand\Qn{\mathbb{Q}}
\newcommand\Zn{\mathbb{Z}}
\newcommand\Nn{\mathbb{N}}

\newcommand{\Sc}[2]{\langle #1,#2\rangle}
\newcommand{\Hom}{\mathrm{Hom}}

\begin{document}

\title{On the Euler characteristic of Kronecker moduli spaces}

\author{Thorsten Weist}
\date{\today}
\email{weist@math.uni-wuppertal.de}
\begin{abstract}
Combining the MPS degeneration formula for the Poincaré polynomial of moduli spaces of stable quiver representations and localization theory, it turns that the determination of the Euler characteristic of these moduli spaces reduces to a combinatorial problem of counting certain trees. We use this fact in order to obtain an upper bound for the Euler characteristic in the case of the Kronecker quiver. We also derive a formula for the Euler characteristic of some of the moduli spaces appearing in the MPS degeneration formula. 
\end{abstract}
\maketitle

\section{Introduction}
\noindent In \cite{mps}, a remarkable formula, called MPS degeneration formula in the following, for the Poincaré polynomial of a smooth compact moduli space of stable quiver representations was derived. Specializing at one this gives a formula for the Euler characteristic. Fixed a dimension vector, this formula reduces the calculation of the Poincaré polynomial (resp. Euler characteristic) to the calculation of an alternating sum of Poincaré polynomials of moduli spaces of a related quiver $\bar Q$ obtained by splitting up all vertices in a certain way. The advantage of this reduction is that only dimension vectors of type one have to be considered, i.e. $d_q\in \{0,1\}$ for all $q\in\bar Q_0$.

In combination with the localization theorem of \cite{wei}, which reduces the calculation of the Euler characteristic of these moduli spaces to torus fixed point components, it follows that the Euler characteristic is already obtained by counting tree shaped subquivers of the universal cover of $\bar Q$ which allow stable representations. The reason for this is that fixed point components under a torus action are given by moduli spaces of the universal covering quiver which is tree shaped. 

In this paper we mostly concentrate on the Kronecker quiver having two vertices and $m\geq 3$ arrows between them. For coprime dimension vectors $(a,b)$, the moduli spaces of stable representations are smooth projective varieties, denoted by $M_{a,b}^s(K(m))$. Based on ideas suggested by M. Douglas, for $b\approx ka$ the Euler characteristic is conjectured to be asymptotically given by 
\[\lim_{a\rightarrow\infty}\frac{1}{a}\ln\chi(M_{a,b}^s(K(m)))=\frac{K}{\sqrt{m-2}}\sqrt{k(m-k)-1}\]
with $K=(m-1)^2\ln((m-1)^2)-(m^2-2m)\ln(m^2-2m)$, see \cite{wei} for more details. 

Even if we cannot verify this conjecture at the moment, in the third section we obtain an upper bound for the Euler characteristic of Kronecker moduli spaces for coprime dimension vectors. To do so, by disregarding stability, we determine an upper bound for the number of tree shaped subquivers of the universal covering which are compatible with some fixed dimension vector. In summary, we obtain     
\[\lim_{a\rightarrow\infty}\frac{1}{a}\ln\chi(M_{a,b}^s(K(m)))\leq(k+1)(\ln(m)+\ln 2+1)-(k-1)\ln k.\] 
Since we have $0<k<m$ if the moduli space is not-empty, fixing $k$ and choosing $m\gg 0$ large enough this coincides with the asymptotic value obtained in \cite{oka}.\\

The fourth section is dedicated to the investigating of localization data, i.e. tuples consisting of a finite subquiver of the universal cover and a dimension vector which corresponds to fixed point components of quiver moduli spaces. We use the fact that the quiver $\overline{K(m)}$ is a complete bipartite quiver and that the stability is given by choosing a level structure on the vertices. Fixed a localization data it turns out that one can recursively split up its vertices, say of level $l$, into vertices of level $1$ and $l-1$ respectively in order to obtain a localization data with the induced level structure. This construction gives a connection between localization data coming along with an arbitrary level structure and localization data with trivial level structure.\\

The dimension vectors appearing in the MPS degeneration formula are given by tuples of weighted partitions of the original dimension vector. In the last section we consider the cases of the partitions $(a=1\cdot a,ka+1=1\cdot (ka+1))$ for which it is possible to determine the Euler characteristic of the moduli spaces $M^s_{(1\cdot a,1\cdot (ka+1))}(\overline{K(m)})$ exactly. If the conjecture of Douglas is true, the result shows that it does not suffice to consider the trivial partition in order to determine the asymptotic value of the Euler characteristic.  
\\\\
{\it Acknowledgements:} I would like to thank S. Okada and M. Reineke for valuable discussions on this topic. I would also like to thank L. Boos and M. Kuschkowitz for inspirational discussions on the combinatorics treated in this paper.

\section{Recollections and Notation}\label{allg}
\noindent In this section, we introduce notation and recollect several results which are important for the remaining part of the paper.  For an introduction to moduli spaces of representations of quivers, we refer to \cite{ReICRA}.

In this paper we restrict to bipartite quivers $Q=(Q_0,Q_1)$ with vertices $Q_0=I\cup J$ and $m(i,j)$ arrows between $i\in I$ and $j\in J$. 
On $\mathbb{Z}Q_0$ we define a (non-symmetric) bilinear form, called the Euler form, by 
\[\langle d,e\rangle:=\sum_{q\in Q_0}d_qe_q-\sum_{\alpha:i\rightarrow j}d_ie_j\]
for $d,\,e\in\mathbb{Z}Q_0$. By \[N_q:=\{q'\in Q_0\mid\exists\alpha:q\rightarrow q'\vee\alpha:q'\rightarrow q\}\]
we denote the set of neighbours of $q\in Q_0$.

For a representation $X$ of the quiver $Q$ we denote by $\underline{\dim} X\in\Nn Q_0$ its dimension vector. Moreover, we choose a level $l:Q_0\rightarrow\Nn_+$ on the set of vertices. Define two linear forms $\Theta, \kappa\in\Hom(\Zn Q_0,\Zn)$ by $\Theta(d)=\sum_{i\in I}l(i)d_i$ and $\kappa(d)=\sum_{q\in Q_0}l(q)d_q$.

Finally, we define a slope function $\mu:\Nn Q_0\backslash\{0\}\rightarrow\Qn$ by \[\mu(d)=\frac{\Theta(d)}{\kappa(d)}.\]
For a representation $X$ of the quiver $Q$ we define $\mu(X):=\mu(\underline{\dim}X)$. A representation $X$ of $Q$ is semistable (resp. stable) if for
all proper subrepresentations $0\neq U \subsetneq X$ the following holds:
\[\mu(U)\leq \mu(X)\text{ (resp. }\mu(U) < \mu(X)).\]
Fixing a slope function as above, we denote by $R^{sst}_d(Q)$ the set of semistable points and by $R^{s}_d(Q)$ the set of stable points in the affine variety $R_d(Q)$ of representations of dimension $d\in\Nn Q_0$. Moreover, let $M^{s}_d(Q)$ (resp. $M^{sst}_d(Q)$) be the moduli space of stable (resp. semistable) representations. In the following we restrict to $\Theta$-coprime dimension vectors, i.e. dimension vectors $d\in\Nn Q_0$ such that $\mu(e)\neq \mu(d)$ for all $0\neq e<d$. In this case semistability and stability coincide, and by \cite{kin} we have that $M^s_d(Q)$ is a smooth projective variety. Note that if $\kappa=\dim$, we obtain the usual definition of stability.\\

In this setup it is easy to check that a representation $X\in R_d(Q)$ is stable if and only if we have
\[\sum_{j\in J} l(j)d'_j>\frac{\sum_{j\in J} l(j)d_j}{\sum_{i\in I} l(i)d_i}\sum_{i\in I} l(i)d'_i\]
for all proper subrepresentations $U\subset X$  where $d'=\underline{\dim} U$.
 
For a vertex $r\in Q_0$ denote by $A_r\subseteq Q_1$ the set of arrows $\alpha$ such that $r$ is a head or tail of $\alpha$. Fixing a vertex $r\in Q_0$ we consider the quiver $Q(r)$ with vertices
\[Q(r)_0=Q_0\backslash r\cup\{r_{l,m}\mid (l,m)\in\Nn_+^2\}\]
and arrows
\begin{eqnarray*}
Q(r)_1&=&Q_1\backslash A_r\cup\{\alpha_1,\ldots,\alpha_l:i\rightarrow r_{l,m}\mid \alpha:i\rightarrow r,\,m\in\Nn_+\}\\&&\cup\{\alpha_1,\ldots,\alpha_l:r_{l,m}\rightarrow j\mid\alpha:r\rightarrow j,\,m\in\Nn_+\}.
\end{eqnarray*}
Moreover, we define the level $l:Q(r)_0\rightarrow\mathbb{N}$ on $Q(r)_0$ by $l(r_{l,m})=l$ whence it coincides with the original one on the remaining vertices. We again consider the stability induced by the new level.

If we fix a dimension vector $d$ and a weighted partition $d_r=\sum_{l=1}^plk_l$, this induces  a dimension vector $\bar{d}$ of $Q(r)$ in the following way: we set $\bar{d}_q=d_q$ for all $q\neq r$ and $\bar d_{r_{l,m}}=1$ for $1\leq l\leq p$ and $1\leq m\leq k_l$ and $\bar d_{r_{l,m}}=0$ otherwise. If we think of a dimension vector of $Q(r)$, from now on we think of a tuple consisting of a dimension vector of $Q$ and weighted partition of $d_r$.

Let $P_d^{Q}(v)$ the Poincaré polynomial in singular cohomology of the moduli space $M^s_d(Q)$. By \cite{mps}, see also \cite{rsw} for a more general setting, we have:
\begin{satz}\label{mpssatz}
For every $\Theta$-coprime dimension vector $d$ we have
\[v^{\Sc{d}{d}-1}P_d^Q(v)=\sum_{\sum lk_l=d_r}P_{\bar{d}}^{Q(r)}(v)v^{\Sc{\bar{d}}{\bar{d}}-1}\left(\prod_l\frac{1}{k_l!}\left(\frac{(-1)^{l-1}(y-y^{-1})}{l(y^l-y^{-l})}\right)^{k_l}\right).\]
\end{satz}
Denoting by $\chi$ the Euler characteristic (in singular cohomology) we get the following corollary:
\begin{kor}\label{mpseuler}We have
\[\chi(M_d^s(Q))=\sum_{\sum lk_l=d_r}\chi(M_{\bar{d}}^s(Q(r)))\prod_l\frac{(-1)^{k_l(l-1)}}{k_l!l^{2k_l}}.\]

\end{kor}
\begin{proof}
We have
\[\frac{(y-y^{-1})}{l(y^l-y^{-l})}=\frac{y^{l-1}(y^2-1)}{l(y^{2l}-1)}=\frac{y^{l-1}(y+1)}{l(\sum_{i=0}^{l-1}y^i)(y^l+1)}.\]
\end{proof}
Clearly, this construction can successively be applied to every vertex. Doing this, the resulting quiver is denoted by $\bar Q$ in the following.
 
Let $\tilde{Q}$ be the universal covering quiver of $Q$. Recall that each vertex of $\tilde Q$ corresponds uniquely to a vertex of $Q$. Denoting by $Q(q)$ those vertices of $\tilde Q$ corresponding to $q\in Q_0$ a dimension vector $\tilde{d}\in\Nn\tilde{Q}_0$ is called compatible with $d\in\Nn Q_0$ if $\sum_{q'\in\mathcal{Q}(q)}\tilde{d}_{q'}=d_q$ for every vertex $q\in Q_0$. By \cite[Corollary 3.14]{wei} we have:
\begin{satz}\label{euler} We have
\[\chi(M^{s}_d(Q))=\sum_{\tilde{d}}\chi(M^{s}_{\tilde{d}}(\tilde{Q})),\]
where $\tilde{d}$ ranges over all equivalence classes being compatible with $d$, and the slope function considered on $\tilde{Q}$ is the one induced by the slope function fixed on $Q$.
\end{satz}
 A tuple consisting of a finite subquiver $\mathcal{Q}$ of the universal covering quiver $\tilde{Q}$ of $Q$ and a dimension vector $d\in\Nn\mathcal{Q}_0$ with $d_q\neq 0$ for all $q\in \mathcal Q_0$ is called localization data if  $M^{s}_{\tilde{d}}(\mathcal{Q})\neq\emptyset$.

\section{An upper bound}\label{upper}
\noindent  In this section, by combining Corollary \ref{mpseuler} and Theorem \ref{euler}, we determine an upper bound for the Euler characteristic of Kronecker moduli spaces, i.e. moduli spaces of stable representations of the generalized Kronecker quiver. In combination with \cite[Section 6]{wei} this shows that the Euler characteristic grows exponentially. Moreover, we compare the result to the conjecture of Douglas.\\

By $K(m)$ we denote the $m$-Kronecker quiver with vertices $i$ and $j$ and $m$ arrows $\alpha_1,\ldots,\alpha_m:i\rightarrow j$. In the following we assume that $m\geq 3$. Recall that for $m\leq 2$ all moduli spaces $M^s_{(a,b)}(K(m))$ are, if not empty, zero- or one-dimensional,  see also the dimension formula in Remark \ref{isom}.

Consider the quiver $\mathcal{M}_m$ defined by the vertices
\[I=\{i_{l,k}\mid (l,k)\in\Nn^2\}\cup\{j\}\]
which has $m\cdot p$ arrows between $i_{p,s}$ and $j$ for all $s$. Consider the stability condition given by the level function $l(i_{l,k})=l$ for all $k$.
Moreover, consider the quiver $\mathcal{N}_m$ defined by the vertices
\[I=\{i_{l,k}\mid (l,k)\in\Nn^2\},\,J=\{j_{l,k}\mid (l,k)\in\Nn^2\}\]
which has $m\cdot p\cdot q$ arrows between $i_{p,s}$ and $j_{q,t}$ for all $s,t$. Consider the stability condition given by the level function $l(i_{l,k})=l(j_{l,k})=l$ for all $k$. \\

We should mention that this definition of stability is actually the one going back to considerations of A. Schofield in \cite{sch} which ensure that the moduli spaces of Schur roots are not empty:
\begin{lem}\label{scho}
A representation of dimension $d$ of $\mathcal{N}_m$ is stable if and only if it is stable in King's sense with the linear form defined by $\Theta_{d}(e):=\Sc{e}{d}-\Sc{d}{e}$. 
\end{lem} 
\begin{proof}Let $X$ be a representation of $\mathcal{N}_m$ of dimension $d$ and let $U$ be a subrepresentation of dimension $e$. Then we have
\[\frac{\sum_{j\in J}l(j)e_j}{\sum_{i\in I}l(i)e_i}>\frac{\sum_{j\in J}l(j)d_j}{\sum_{i\in I}l(i)d_i}\]
if and only if
\[\sum_{(i,j)\in I\times J}l(j)l(i)d_ie_j>\sum_{(i,j)\in I\times J}l(j)l(i)e_id_j.\]
Since we have $m(i,j)=ml(i)l(j)$ the claim follows.\end{proof}
 
\begin{bem}\label{isom}
\end{bem}
\begin{itemize}
\item By the results of \cite{sch} together with Lemma \ref{scho} we have that $M^s_{a,b}(K(m))\neq\{pt\}$ if and only if $(a,b)$ is an imaginary Schur root. By \cite{kac}, we have that $(a,b)$ is an imaginary Schur root if and only if
\[\frac{\textstyle m-\sqrt{m^2-4}}{\textstyle
2}<\frac{\textstyle b}{\textstyle a}<\frac{\textstyle
m+\sqrt{m^2-4}}{\textstyle 2}\text{ holds.}\]
In the following we only consider dimension vectors $(a,b)$ such that these inequalities hold. We then have $\dim\,M^s_{a,b}(K(m))=1-\Sc{(a,b)}{(a,b)}=1-a^2-b^2+abm.$
\item We will frequently make use of the well-known isomorphisms of Kronecker moduli spaces $M^s_{(a,b)}(K(m))\cong M^s_{(a,b)}(K(m))$ and $M^s_{(a,b)}(K(m))\cong M^s_{(a,ma-b)}(K(m))$. They are induced by the isomorphisms of the representation spaces obtained by taking the transpose of representations and the reflection functor introduced in \cite{bgp} respectively.\\
\end{itemize}

Every pair of weighted partitions $(a=\sum_{l}la_l,b=\sum_llb_l)$ defines a dimension vector $d$ of $\mathcal{N}_m$ by setting $d_{i_{l,k}}=1$ for $k=1,\ldots,a_l$ and $d_{i_{l,k}}=0$ otherwise (resp. $d_{j_{l,k}}=1$ for $k=1,\ldots,b_l$ and $d_{j_{l,k}}=0$ otherwise). In the following, we denote this dimension vector by $\overline{(a,b)}$. Applying Theorem \ref{mpseuler} to both vertices, for coprime $(a,b)$, we obtain
\[\chi(M_{a,b}^s(K(m)))=\sum_{\substack{\sum la_l=a\\\sum lb_l=b}}\chi(M_{\overline{(a,b)}}^s(\mathcal{N}_m))\prod_l\frac{(-1)^{(a_l+b_l)(l-1)}}{a_l!b_l!l^{2(a_l+b_l)}}.\]

In the following, all dimension vectors $(a,b)$ of $K(m)$ are assumed to be coprime. Fixed a weighted partition $a=\sum_{l}la_l$ define $\hat a=\sum_{l}a_l$ and  $\tilde a=a-\hat a=\sum_l(l-1)a_l$. Starting with this formula the next step is to apply the localization theorem to the moduli spaces $M_{\overline{(a,b)}}^s(\mathcal{N}_m)$.\\

Fixed a pair of weighted partitions $(a=\sum_lla_l,b=\sum_llb_l)$, we consider the quiver $Q(\sum_lla_l,\sum_llb_l)$ with labelled vertices $I\cup J$ with $I=\bigcup_{k=1}^aI_k$, $J=\bigcup_{k=1}^bJ_k$ and $|I_k|=a_k,\,|J_k|=b_k$ and, moreover, having $mln$ arrows going from $i$ to $j$ whenever $i\in I_l$ and $j\in J_n$. The stability is given by the level defined by $l(i)=k$ for every $i\in I_k$ and $l(j)=k$ for every $j\in J_k$. This quiver is just the support of the pair of fixed weighted partitions understood as dimension vector of $\mathcal{N}_m$.

Each localization data $(\mathcal{Q},d)$  with sources $I$ and sinks $J$ such that $d$ is compatible with $\overline{(a,b)}$ corresponds to a connected subtree of $\tilde{\mathcal{N}}_m$. Actually this subtree is already obtained from $Q(\sum_lla_l,\sum_llb_l)$ by deleting certain arrows because we have $d_q=1$ for all $q\in\mathcal{Q}_0$. Because of this we also have $\chi(M^s_{\overline{(a,b)}}(\mathcal{Q}))=1$ because the moduli space is a point. 
\begin{bem}\label{bem1}
\end{bem}
\begin{itemize}
\item Every localization data $(\mathcal Q,d)$ comes with a colouring of the arrows $c:\mathcal{Q}_1\rightarrow Q_1$. If we forget about this colouring of the arrows of some localization data $(\mathcal Q,d)$, with each vertex $q\in \mathcal{Q}_0$ we can associate the number
\[v(q):=\prod_{q'\in N_q}l(q').\] Define $v((\mathcal{Q},d)):=\prod_{q\in\mathcal{Q}_0}v(q).$ Note that, fixed an uncoloured localization data a colouring can also be understood as an embedding into the universal cover. 

Thus if $(Q,d)$ is an uncoloured localization data  of $\mathcal{N}_m$ compatible with $(a=\sum_lla_l,b=\sum_llb_l)$ the number of different colourings of the arrows is given by
\[m^{\hat a+\hat b-1}v((\mathcal{Q},d)).\]
Indeed, every such localization data is forced to have $\hat a+\hat b-1$ arrows.  If we denote the set of uncoloured localization data compatible with this pair by $\mathcal{L}_{(a,b)}(\mathcal{N}_m)$ we obtain
\[\chi(M^s_{\overline{(a,b)}}(\mathcal{N}_m))=m^{\hat a +\hat b-1}\sum_{(\mathcal{Q},d)\in \mathcal{L}_{(a,b)}(\mathcal{N}_m)}v((\mathcal{Q},d)).\]
\item As far as the quiver $\mathcal{N}_m$ is concerned we only consider dimension vectors of type one. Thus every localization data is also of type one. Therefore, it is uniquely determined by its quiver and we will sometimes skip the dimension vector.\\
\end{itemize}

\begin{bei}
\end{bei}
\noindent Consider the dimension vector $(2,3)$. Then we have to consider the uncoloured localization data
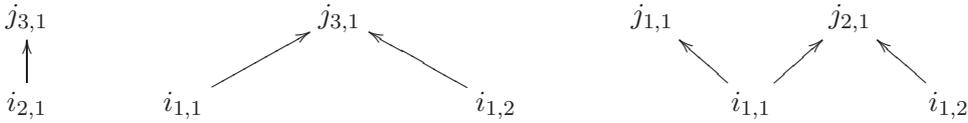
\[
\begin{xy}
\xymatrix@R16pt@C15pt{
j_{3,1}&&&&j_{3,1}&&&&j_{1,1}&&j_{2,1}\\i_{2,1}\ar[u]&&i_{1,1}\ar[rru]&&&&i_{1,2}\ar[llu]&&&i_{1,1}\ar[lu]\ar[ru]&&i_{1,2}\ar[lu]}
\end{xy}
\]\\
corresponding to the pairs of partitions $(2\cdot 1,3\cdot 1),\,(1\cdot 2,3\cdot 1),\,(1\cdot 2,1\cdot 1+2\cdot 1)$ and the uncoloured localization data
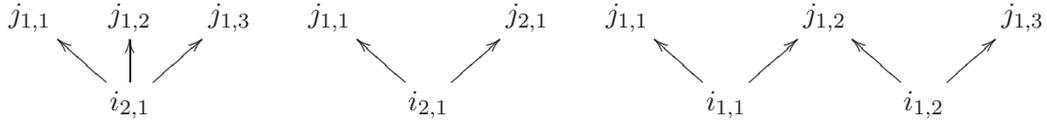
\[
\begin{xy}
\xymatrix@R16pt@C15pt{
j_{1,1}&j_{1,2}&j_{1,3}&j_{1,1}&&j_{2,1}&j_{1,1}&&j_{1,2}&&j_{1,3}\\&i_{2,1}\ar[u]\ar[ru]\ar[lu]&&&i_{2,1}\ar[lu]\ar[ru]&&&i_{1,1}\ar[lu]\ar[ru]&&i_{1,2}\ar[lu]\ar[ru]}
\end{xy}
\]
corresponding to $(2\cdot 1,1\cdot 3),\,(2\cdot 1,1\cdot 1+2\cdot 1),\,(1\cdot 2,1\cdot 3)$. Then we get 
\[\chi(M^s_{2,3}(K(m)))=-\frac{m}{6}+\frac{m^2}{2}-m^3-\frac{1}{3}m^3+\frac{m^2}{2}+\frac{1}{2}m^4=\frac{1}{2}m^4-\frac{4}{3}m^3+m^2-\frac{m}{6}.\]
\\
With a connected (multi)graph $G=(V,E)$ with vertices $V$ and edges $E$ we can associate the number $\tau(G)$ of its spanning trees, i.e. subtrees of $G$ involving all vertices. Recall that the degree of a vertex $v\in V$ is the number of its incident edges. Each localization data compatible with a pair of weighted partitions $(a=\sum_lla_l,b=\sum_llb_l)$ defines a spanning tree of $Q(\sum_lla_l,\sum_llb_l)$. Thus the problem of finding an upper bound may be reduced to counting the number of spanning trees of $Q(\sum_lla_l,\sum_llb_l)$ for all weighted partitions of $a$ and $b$. The following two results are very useful for our purposes, see \cite{sco} and \cite{tho}:
\begin{satz}\label{spanning} \begin{enumerate}
\item Let $K_{a,b}:=(I+J,E)$ be the complete bipartite graph with $|I|=a$ and $|J|=b$, i.e. the graph having vertices $I\cup J$ and edges $(i,j)$ for all $i\in I,\,j\in J$. Then we have $\tau(K_{a,b})=a^{b-1}b^{a-1}$.
\item For the number of spanning trees $\tau(G)$ of a multigraph $G$ with vertices $q_1,\ldots, q_n$ of degrees $d_1,\ldots,d_n$ we have
$\tau(G)\leq d_1\ldots d_{n-1}.$
\end{enumerate} 
\end{satz}
Note that $K_{a,b}$ is just $Q(1\cdot a,1\cdot b)$ when forgetting the orientation of the arrows. From this we get the following corollary:
\begin{kor}
Let $(\sum_lla_l,\sum_llb_l)$ be a pair of weighted partitions. Then we have 
\[\chi(M^s_{(1\cdot a,1\cdot b)}(\mathcal{N}_m))\leq m^{a+b-1}a^{b-1}b^{a-1}\]
and 
 \[\chi(M^s_{\overline{(a,b)}}(\mathcal{N}_m))\leq m^{\hat a+\hat b-1}b^{\hat a}a^{\hat b}\prod_ll^{a_l+b_l}.\]
\end{kor}
\begin{proof} The first statement follows by the first part of the previous theorem together with Remark \ref{bem1}. In general, every source of level $l$ has degree $l\cdot b\cdot m$ and every sink of level $l$ has degree $l\cdot a\cdot m$. Since there exist $a_l$ sources (resp. $b_l$ sinks) of level $l$ in $Q(\sum_lla_l,\sum_llb_l)$ and, therefore, $\hat a$ sources (resp. $\hat b$ sinks) in total, the claim follows by the second part of the preceding theorem in the same way as the first part. Note that the product on the right hand side in the second part of Theorem \ref{spanning} does not involve the degrees of all vertices.
\end{proof} 
In order to treat the multinomial coefficients appearing in the MPS degeneration formula, we make use of the following lemma:
\begin{lem}
Fix $a,\,\hat{a}\in\Nn$. Let 
\[\mathcal{A}(a,\hat a,s):=\{(a_1,\ldots,a_s)\in\Nn_+^s\mid \sum_l a_l=\hat a,\,\sum_lla_l=a\}\]
and $\mathcal{A}(a,\hat a)=\bigcup_s\mathcal{A}(a,\hat a,s)$.
Then we have
\[\sum_{(a_1,\ldots,a_s)\in\mathcal{A}(a,\hat a)}\frac{\hat a!}{a_1!\ldots a_s!}=\binom{a-1}{\hat a-1}.\]

\end{lem}
\begin{proof} Consider the set $S=\{q_1,\ldots,q_{\hat{a}}\}$ and let $T(\hat a,a):=\{l:S\rightarrow\Nn_+\mid\sum_{k=1}^{\hat a}l(q_k)=a\}$. 
Fixed a tuple $(a_1,\ldots,a_s)\in\mathcal{A}(a,\hat a,s)$ the corresponding summand is the number of level structures $l:S\rightarrow \Nn_+\in T(\hat a,a)$ such that $a_i$ elements have level $i$. Each such choice defines a graph with $a+1$ vertices
\[I=\{i_{k,l}\mid k=1,\ldots,\hat a,\,l=1,\ldots,l(i_k)\}\cup \{j\}\] and edges $(i_{k,1},j)$. Note that there always exists an edge $(i_{1,1},j)$. We denote all such graphs by $G(\hat a,a)$. The number of all such graphs is given by the right hand side of the formula. The other way around every such graph defines a level function $l:S\rightarrow\Nn_+\in T(\hat a,a)$ because there exist exactly  $a$ vertices and $\hat a$ edges. More detailed let $I=\{i_k\mid k=1,\ldots,a\}$ have edges $(i_{k_1},j),\ldots,(i_{k_{\hat a}},j)$ with $k_1<k_2<\ldots<k_{\hat a}$. Then we define $l(q_n)=k_{n+1}-k_{n}$ for $n\leq\hat a-1$ and $l(q_{\hat a})=a-k_{\hat a}+1$. Now it is straightforward to check that this gives a bijection between $T(\hat a,a)$ and $G(\hat a ,a)$. 
\end{proof}
We make use of the following well-known lemma:
\begin{lem}
For every weighted partition $a=\sum_lla_l$ with $a\geq 1$ we have
\[\frac{1}{\prod_ll^{a_l}}\binom{a}{\hat a}\leq 2^a.\]
\end{lem}
\begin{proof}
This follows from
$2^a=\sum_{i=0}^a\binom{a}{i}\geq\binom{a}{\hat a}$
for $0\leq \hat a\leq a$.

\end{proof}

\begin{bem}
\end{bem}
\begin{itemize}
\item Recall that $\frac{1}{m}a\leq b\leq ma$ is no restriction because every Schur root of the Kronecker quiver satisfies this condition. 
\item Actually by the applied methods it seems that one cannot avoid some factor like $C^{a+b}$ with $1< C\leq 2$ in the upper bound. The reason for this is that
\[\frac{a!}{\hat a!}\binom{a}{\hat a}\nleq (Ka)^{\tilde a}\]
for some fixed constant $K>0$ and for all $a,\hat a\in\Nn$.

\end{itemize}

The last ingredient is an upper bound for the number of partitions of a given number $a\in\Nn$, see for instance \cite[Section 6]{kno}:
\begin{lem}
The number of partitions of $a$ is bounded by $\exp(\pi\sqrt{\frac{2a}{3}})$.
\end{lem}

In summary we obtain an upper bound for the Euler characteristic:
\begin{satz}\label{upperbound}
We have
\[\chi(M^s_{a,b}(K(m))\leq \frac{1}{a!b!}2^{a+b}m^{a+b-1}\exp(\pi\sqrt{\frac{2}{3}}(\sqrt{a}+\sqrt{b}))b^{a+1/2}a^{b+1/2}.\]
\end{satz}
\begin{proof}
Applying the results of this section we obtain
\begin{eqnarray*}\chi(M_{a,b}^s(K(m)))&=&\sum_{\substack{\sum la_l=a\\\sum lb_l=b}}\prod_l\frac{(-1)^{(a_l+b_l)(l-1)}}{a_l!b_l!l^{2(a_l+b_l)}}\chi(M_{\overline{(a,b)}}^s(\mathcal{N}_m))\\
&\leq& \sum_{\substack{\sum la_l=a\\\sum lb_l=b}}\frac{1}{\hat a!\hat b!}\binom{a}{\tilde a}\binom{b}{\tilde b}\frac{\chi(M_{\overline{(a,b)}}^s(\mathcal{N}_m))}{\prod_ll^{2(a_l+b_l)}}\\
&\leq& \sum_{\substack{\sum la_l=a\\\sum lb_l=b}}2^{a+b}\frac{1}{a!b!}b^{\tilde a+1/2}a^{\tilde b+1/2}m^{\tilde a+\tilde b}\frac{\chi(M_{\overline{(a,b)}}^s(\mathcal{N}_m))}{\prod_ll^{a_l+b_l}}\\
&\leq& \sum_{\substack{\sum la_l=a\\\sum lb_l=b}}2^{a+b}\frac{1}{a!b!}b^{\tilde a+1/2}a^{\tilde b+1/2}m^{\tilde a+\tilde b}b^{\hat a}a^{\hat b}m^{\hat a+\hat b-1}\\
&\leq&\sum_{\substack{\sum la_l=a\\\sum lb_l=b}}2^{a+b}\frac{1}{a!b!}b^{a+1/2}a^{b+1/2}m^{a+ b-1}\\
&\leq&\frac{1}{a!b!}2^{a+b}m^{a+b-1}\exp(\pi\sqrt{\frac{2}{3}}(\sqrt{a}+\sqrt{b}))b^{a+1/2}a^{b+1/2}.
\end{eqnarray*}
\end{proof}
\begin{bem}\label{bem2}
\end{bem}
\begin{itemize}
\item Note that if we only consider the summand corresponding to the trivial partition the term $\ln 2$ vanishes. 
\item Using the isomorphisms of moduli spaces we can improve the upper bound  obtained in Theorem \ref{upperbound}. Indeed, it is straightforward to check that the bound is sharpest if $\frac{2}{m}a\leq b\leq\frac{m}{2}a$ which can be assumed by the isomorphisms.
\end{itemize}
Let $b\approx ka$ such that $(a,b)$ coprime. Let 
$m_1=\frac{\textstyle m-\sqrt{m^2-4}}{\textstyle
2}$ and $m_2:=\frac{\textstyle
m+\sqrt{m^2-4}}{\textstyle 2}$. Define $f_m:[m_1,m_2]\rightarrow\mathbb{R}$ by 
\[f_m(k):=\lim_{a\rightarrow\infty}\frac{1}{a}\ln\chi(M_{a,b}^s(K(m))).\]
As already mentioned in the introduction it is conjectured that
\[f_m(k)=\frac{K}{\sqrt{m-2}}\sqrt{r(m-r)-1}\]
with $K=(m-1)^2\ln((m-1)^2)-(m^2-2m)\ln(m^2-2m)$, see \cite[Section 6]{wei} for more details. We can compare $f_m$ to the upper bound using the following corollary:  
\begin{kor}\label{ln2}
Let $b\approx ka$ and let $(a,b)$ coprime. Then we have 
\[\lim_{a\rightarrow\infty}\frac{1}{a}\ln\chi(M_{a,b}^s(K(m)))\leq(k+1)(\ln(m)+\ln 2+1)-(k-1)\ln k.\]
\end{kor}
\begin{proof}
Since we are interested in the logarithmic behaviour, we make use of the Stirling formula when setting $x!\approx \frac{x^x}{\exp(x)}$. For $a\gg 0$ we obtain
\[
\frac{(ka)^aa^{ka}}{a!(ka)!}\approx\frac{(ka)^aa^{ka}\exp((k+1)a)}{a^a(ka)^{ka}}=\frac{\exp((k+1)a)}{k^{(k-1)a}}.
\]
Thus we get
\begin{eqnarray*}
\lim_{a\rightarrow\infty}\frac{1}{a}\ln\chi(M_{a,b}^s(K(m)))&\leq&\lim_{a\rightarrow\infty}\frac{1}{a}\ln\frac{e^{\pi\sqrt{2/3}(\sqrt{a}+\sqrt{ka})}2^{(k+1)a}m^{(k+1)a-1}(ka)^{a+1/2}a^{ka+1/2}}{a!(ka)!}
\\&=&(k+1)(\ln(m)+\ln 2+1)-(k-1)\ln k.
\end{eqnarray*}
\end{proof}
In \cite{wei} it is shown that the Euler characteristic grows at least exponentially. Together with this result we obtain that the Euler characteristic of Kronecker moduli spaces grows exponentially.\\

Define $g_m:[m_1,m_2]\rightarrow\mathbb{R}$ by $g_m(k):=(k+1)(\ln(m)+\ln(2)+1)-(k-1)\ln k$. It is straightforward to check that we have $h_m(k):=\frac{g_m(k)}{f_m(k)}>1$ for $k\in [m_1,m_2]$. More detailed, one finds out that $h_m$ has a minimum at $k=1$. 

\section{On the recursive construction of localization data}
\noindent In order to determine the Euler characteristic of Kronecker moduli spaces using the MPS degeneration formula together with the localization theorem one has to determine all localization data $(\mathcal{Q},d)$ where $d$ is compatible with a tuple of partitions of a fixed dimension vector $(a,b)$. The method described in this section shows that every localization data compatible with a partition of $(a,b)$ corresponds to one compatible with the trivial partition of $(a,b)$. To do so, starting with a fixed partition, we state a method how to construct localization data of a refined partition recursively.  \\

Let $(\mathcal Q,d)$ be an uncoloured localization data which is compatible with the weighted partition $(a=\sum_lla_l,b=\sum_llb_l)$ understood as a dimension vector of $\mathcal{N}_m$. If this weighted partition is non-trivial with $a_{k}\neq 0$, we modify it at some sink by defining
\[a':=\sum_{l}a'_ll\]
by $a'_1=a_1+1$, $a'_{k-1}=a_{k-1}+1$, $a'_{k}=a_{k}-1$ and $a'_l=a_l$ for $l\notin\{1,k-1,k\}$. 

Consider a source $i_{k}\in\mathcal{Q}$ of level $k$. Then we split up this vertex into two vertices $i_{k-1}$ and $i_{1}$ of level $k-1$ and $1$ respectively. Moreover, let $d_{i_{k-1}}=1=d_{i_{1}}$ and let $\hat{J}=\{j_1,\ldots,j_s\}= N_{i_{k}}$ of level $l_1,\ldots,l_s$. The stability condition implies  
\[\frac{a}{a+b}>\frac{k}{k+\sum_{i=1}^sl_i}.\]
Using this notation we have the following lemma:
\begin{lem}
There exists a decomposition $\hat{J}=J_1\cup J_2$ with $J_1\cap J_2=\{j_t\}$ such that, setting $N_{i_{1}}=J_1$, $N_{i_{k-1}}=J_2$ and leaving the remaining quiver the way it is, the resulting data is a localization data compatible with the weighted partition $(a',b)$.
\end{lem}
\begin{proof}
It suffices to show that for every sequence of positive integers $(l_1,\ldots,l_s)$ there exists a decomposition such that
\[\frac{k}{k+\sum_{i=1}^sl_i}\geq \frac{k-1}{k-1+\sum_{j\in J_1}l_j}\text{ and } \frac{k}{k+\sum_{i=1}^sl_i}\geq \frac{1}{1+\sum_{j\in J_2}l_j}.\]
These inequalities are easily seen to be equivalent to
\[(k-1)\sum_{i=1}^sl_i\leq k\sum_{i=t}^{s}l_i\text{ and } \sum_{i=1}^sl_i\leq k\sum_{i=1}^tl_i.\]
For the proof we keep in mind that the slope of the localization data does not change and that we just modify one sink. This means that parts of the modified data which do not include the modified sinks do not contradict the stability condition. Moreover, if these inequalities are satisfied, parts of the data including the modified sinks are easily seen to be of smaller slope than the corresponding parts of the original localization data.

We proceed by induction on $|J|$. For $|J|=1$ the statement is trivial. Assume that for $J_1=\{j_1,\ldots,j_t\}$ and $J_2=\{j_t,\ldots,j_{s}\}$ the inequalities from above hold and let $j_{s+1}$ be an additional sink. First assume that $(k-1)\sum_{i=1}^tl_i>\sum_{i=t+1}^{s+1}l_i$. Then 
\[\sum_{i=1}^{s+1}l_i\leq k\sum_{i=1}^tl_i\]
follows by this inequality and the other inequality follows by the induction hypothesis because $(k-1)l_{s+1}<kl_{s+1}$. Thus assume that $(k-1)\sum_{i=1}^tl_i\leq \sum_{i=t+1}^{s+1}l_i$. Then we have
\[(k-1)\sum_{i=1}^{s+1}l_i\leq k\sum_{i=t+1}^{s+1}l_i.\]
Moreover, since $\sum_{i=1}^sl_i\leq k\sum_{i=1}^tl_i$ it follows that
\[\sum_{i=1}^{s+1}l_i\leq k\sum_{i=1}^tl_i+kl_{s+1}.\]
Note that, in the second case $j_{s+1}$ is the common sink.
\end{proof}

Applying this method recursively this shows that with every localization data which is compatible with an arbitrary partition we can associate a localization data which is compatible with the trivial partition. Unfortunately this construction is not unique. On the one hand there can be more than one modified localization data which is no problem as long as we are only interested in an upper bound. But, on the other hand, there can be two localization data such that their modified data coincide. For instance consider the localization data given by the quivers
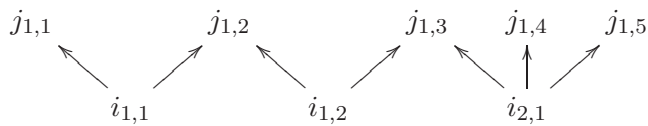
\[
\begin{xy}
\xymatrix@R16pt@C15pt{j_{1,1}&&j_{1,2}&&j_{1,3}&j_{1,4}&j_{1,5}\\
&i_{1,1}\ar[lu]\ar[ru]&&i_{1,2}\ar[lu]\ar[ru]&&i_{2,1}\ar[lu]\ar[u]\ar[ru]
}\end{xy}
\]
and
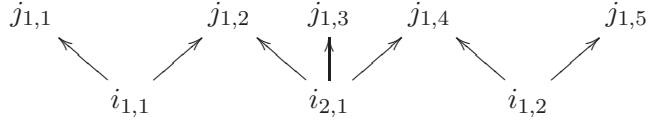
\[
\begin{xy}
\xymatrix@R16pt@C15pt{j_{1,1}&&j_{1,2}&j_{1,3}&j_{1,4}&&j_{1,5}\\
&i_{1,1}\ar[lu]\ar[ru]&&i_{2,1}\ar[u]\ar[lu]\ar[ru]&&i_{1,2}\ar[lu]\ar[ru]
}\end{xy}
\]
with level structure given by $l(q_{l,k})=l$ for $q\in\{i,j\}$. Applying the methods we described, in both cases one of the modified localization data is
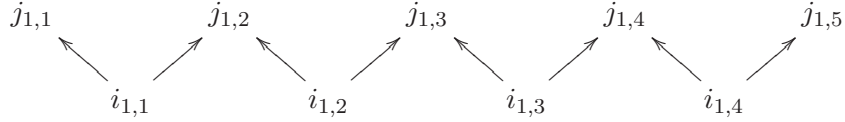
\[
\begin{xy}
\xymatrix@R16pt@C15pt{j_{1,1}&&j_{1,2}&&j_{1,3}&&j_{1,4}&&j_{1,5}\\
&i_{1,1}\ar[lu]\ar[ru]&&i_{1,2}\ar[lu]\ar[ru]&&i_{1,3}\ar[lu]\ar[ru]&&i_{1,4}\ar[lu]\ar[ru]
}\end{xy}
\]
Nevertheless, it gives a connection between localization data compatible with an arbitrary partition and the one being compatible with the trivial partition.\\

\begin{bem}
\end{bem}
\begin{itemize}
\item One might also apply the MPS degeneration formula only to the source (resp. sink) of $K(m)$. The resulting infinite quiver is the quiver $\mathcal{M}_m$ introduced in Section \ref{upper}. In this case there exist $l\cdot m$ arrows from a source $i_{l,k}$ to $j$. Thus, fixing a localization data, every source of level $l$ has less than $l\cdot m$ neighbours. This restricts the number of neighbours of some localization data. A construction similar to the one introduced in this section applies.

An advantage of applying the MPS-formula only to the source is that the number of neighbours of sources in some localization data is bounded by the number of outgoing arrows of the source it corresponds to. More detailed, a source of level $l$ has at most $l\cdot m$ neighbours. The disadvantage is that the moduli spaces of localization data are no points in general.\\
\item Summarizing, the next step could be to think about the following questions:
\begin{enumerate}
\item Can one modify this construction in order to make it unique?
\item What are the fibres of the corresponding map?\\
\end{enumerate}
Answering this questions, should help to prove Douglas' conjecture because combinatorics clearly simplify when only considering the trivial partition.
\end{itemize}
\section{A formula for the Euler characteristic of certain moduli spaces}
\noindent In this section we consider the quiver $\mathcal{N}_m$ and the trivial partition $(1\cdot a,1\cdot (ka+1))$ for some $k\in\Nn_+$. We obtain a formula for the Euler characteristic of the moduli space $M^s_{(1\cdot a,1\cdot (ka+1))}(\mathcal{N}_m)$. This formula holds for arbitrary positive integers $m$ and $k$.
\begin{lem}\label{lemmaaka}
Let $(\mathcal Q,d)$ be a localization data which is compatible with $(1\cdot a,1\cdot (ka+1))$. Then $\mathcal Q$ is a tree with $a$ sources and $ka+1$ sinks such that every source has exactly $k+1$ neighbours. 
\end{lem}
\begin{proof}
By the stability condition for every source $i\in\mathcal{Q}_0$ we have $|N_i|\geq k+1$. Now we can be proceed by induction on the number of sources $a$ using that every localization data has a subquiver $(i,j_1,\ldots,j_n)$ such that $N_{j_s}=\{i\}$ for all but one $s$. If we had $n\geq k+2$ the remaining part of the data would contradict the stability condition because
\[k(a+1)+1-(k+1)=ka<\frac{k(a+1)+1}{a+1}a=ka+\frac{a}{a+1}.\]
Thus we have $n=k+1$. Deleting this subquiver except the sink with $|N_{j_k}|>1$ it is straightforward to check that we obtain a localization data compatible with $(a,ka+1)$ because
\[\frac{k(a+1)+1}{a+1}a'<b'\Leftrightarrow\frac{ka+1}{a}a'<b'\]
for all $a'<a$ and $b'\in\Nn$. Thus the claim follows by the induction hypotheses.
\end{proof}

Let $(\mathcal{Q},d)$ a localization data compatible with $(1\cdot a,1\cdot (ka+1))$. Since the dimension vector is already given by $\mathcal{Q}$ in abuse of notation we will skip $d$ in what follows. Forgetting about the colouring of the vertices and arrows we can assign to $\mathcal{Q}$ the weight $w(\mathcal{Q})=\frac{1}{|\mathrm{Aut}(\mathcal{Q})|}$ of quiver automorphisms. Define
\[\mathcal T(a,ka+1):=\sum_{\mathcal{Q}}w(\mathcal Q)\]
where the sum is taken over all uncoloured localization data compatible with $(1\cdot a,1\cdot(ka+1))$.
\begin{pro}
We have
\[\mathcal T(a,ka+1)=\frac{1}{(ka+1)^2}\frac{1}{a!}\left(\frac{(ka+1)}{k!}\right)^{a}.\]
\end{pro}
\begin{proof}
To construct such uncoloured localization data recursively (taking into account the quiver symmetries) we start with a single sink $j$ which is assigned to be the root. Then we can glue arbitrarily many subquivers of the form $(i,j_1,\ldots,j_k)$ to this and successively to every resulting sink.
  
Let $y(x)$ be the generating function of rooted uncoloured localization data  compatible with $(1\cdot a,1\cdot(ka+1))$ taking into account quiver automorphism and consider
\[\Phi(x):=\exp(\frac{x^k}{k!})=\sum_{n=0}^{\infty}\frac{1}{n!}\left(\frac{x^{k}}{k!}\right)^n.\]
Then $y(x)$ satisfies the functional equation $y(x)=x\Phi(y(x))$. By the Lagrangian inversion formula, see for instance \cite[Section 5.4]{sta}, we thus have:
\[[x^t]y(x)=\frac{1}{t}[u^{t-1}]\Phi(u)^t=\frac{1}{t}[u^{t-1}]\sum_{n=0}^{\infty}\frac{1}{n!}\left(t\frac{u^{k}}{k!}\right)^n=\frac{1}{t}\left\{\begin{matrix}
\frac{1}{\left(\frac{t-1}{k}\right)!}\left(\frac{t}{k!}\right)^{\frac{t-1}{k}}\,\textrm{if }k|(t-1)\\
0\,\textrm{ otherwise} 
\end{matrix}\right.\]
where $[x^t]y(x)$ denotes the $t$-th coefficient of the power series $y(x)$. Thus we have
\[[x^{ka+1}]y(x)=\frac{1}{(ka+1)}\frac{1}{a!}\left(\frac{(ka+1)}{k!}\right)^{a}.\]

Every constructed graph contains a sink which is assigned to be the root vertex. Since every graph has $ka+1$ sinks, the result follows. 
\end{proof}

Since there exist $a!$ possibilities to label the sources and $(ka+1)!$ possibilities to label the sinks we obtain the following corollary:
\begin{kor}
Let $G$ be the complete bipartite graph having $a$ labelled sources and $ka+1$ labelled sinks. Then there exist \[\frac{(ka)!}{ka+1}\left(\frac{ka+1}{k!}\right)^{a}\] spanning trees such that every source has exactly $k+1$ incident edges.
\end{kor}
It might be possible that this formula is already known, but it could not be found in the literature. For $k=1$ the resulting sequence appears as sequence A066319 in \cite{slo}. It counts labelled structures which are simultaneously trees and cycles, see \cite[Section 2.1]{bll}.
   
Using Theorem \ref{euler}, Remark \ref{bem1} and the results of this section we have:
\begin{satz}
We have
\[\chi(M^s_{(1\cdot a,1\cdot(ka+1))}(\mathcal{N}_m))=m^{(k+1)a}\frac{(ka)!}{ka+1}\left(\frac{ka+1}{k!}\right)^{a}\]
\end{satz}
We are also interested in the contribution of the summand corresponding to the trivial partition to the Euler characteristic of Kronecker moduli spaces:
\begin{kor}
We have
\[\lim_{a\rightarrow\infty}\frac{1}{a}\ln\left(\frac{\chi(M^s_{(1\cdot a,1\cdot (ka+1))}(\mathcal{N}_m)}{a!(ka+1)!}\right)=\ln(m)(k+1)+1-\ln((k-1)!).\]
\end{kor}
\begin{proof}
Using the Stirling formula for $a\gg 0$ we have
\[\frac{1}{a!(ka+1)!}\frac{(ka)!}{ka+1}\left(\frac{ka+1}{k!}\right)^a\approx\frac{\exp(a)}{(k!)^a}\left(\frac{ka+1}{a}\right)^a\\
=\frac{\exp(a)}{((k-1)!)^a}.
\]
\end{proof}
Define $i_m(k):=\ln(m)(k+1)+1-\ln((k-1)!)$. One can check that $g_m(k)>i_m(k)>f_m(k)$ for fixed $m$. This also means that, if the Douglas' conjecture is true, it does not suffice to consider the summand corresponding to the trivial partition in order to investigate the asymptotic behaviour of the Euler characteristic exactly.

\end{document}